\documentclass[12pt]{amsart}
\usepackage{amsmath,amsthm,amssymb,mathrsfs}
\usepackage{graphicx,graphics,pdfsync}
\usepackage{eucal,esint}
\numberwithin{equation}{section}
\usepackage{epstopdf}
\usepackage{color}
\usepackage[a4paper,margin=3.00truecm]{geometry}
\newcounter{marnote}

\def \dis {\displaystyle}

\def \into {\int_\Omega}
\def \confai {-\kern -.5em\rightharpoonup}
\def \cqfd {\hfill$\Box$}

\def \al {\alpha}
\def \be {\beta}
\def \ga {\gamma}

\def \ep {\varepsilon}

\def \Om {\Omega}
\def \la {\lambda}

\def \RR {\mathbb R}
\def\A{\mathcal{A}}

\def\E{\mathcal{E}}

\def \M {\mathscr{M}}

\def \beq {\begin{equation}}
\def \eeq {\end{equation}}
\def \ba {\begin{array}}
\def \ea {\end{array}}

\def \ecart {\noalign{\medskip}}

\DeclareMathOperator{\dive}{div}
\DeclareMathOperator{\spt}{spt}
\newtheorem{Thm}{Theorem}[section]

\newtheorem{Adef}[Thm]{Definition}

\newtheorem{Arem}[Thm]{Remark}
\newenvironment{Rem}{\begin{Arem}\rm}{\end{Arem}}
\newtheorem{Aexa}[Thm]{Example}
\newenvironment{Exa}{\begin{Aexa}\rm}{\end{Aexa}}
\newtheorem{Anot}[Thm]{Notation}

\def \reff #1.{figure~\ref{#1}}
\def \refs #1.{Section~\ref{#1}}
\def \refss #1.{Subsection~\ref{#1}}
\def \refD #1.{Definition~\ref{#1}}
\def \refT #1.{Theorem~\ref{#1}}
\def \refL #1.{Lemma~\ref{#1}}
\def \refC #1.{Corollary~\ref{#1}}
\def \refP #1.{Proposition~\ref{#1}}
\def \refPt #1.{Properties~\ref{#1}}
\def \refR #1.{Remark~\ref{#1}}
\def \refE #1.{Example~\ref{#1}}
\def \refN #1.{Notation~\ref{#1}}

\title{{\bf Optimal coefficients for elliptic PDEs}}

\begin{document}
\maketitle
{\it Dedicated to the memory of Hedy Attouch}\bigskip

\centerline{G. Buttazzo$^\dag$,\hskip 0.3cm J. Casado-D\'{\i}az$^{\dag\dag}$,\hskip 0.3cm F. Maestre$^{\dag\dag}$}
\bigskip 
\centerline{{$^\dag$} Dipartimento di Matematica, Universit\`a di Pisa,}
\centerline{Largo B. Pontecorvo, 5}
\centerline{56127 Pisa, ITALY}
\bigskip
\centerline{$^{\dag\dag}$ Dpto. de Ecuaciones Diferenciales y An\'alisis Num\'erico,}
\centerline{Facultad de Matem\'aticas, C. Tarfia s/n}
\centerline{41012 Sevilla, SPAIN}
\bigskip
\centerline{e-mail: giuseppe.buttazzo@unipi.it, jcasadod@us.es,
fmaestre@us.es}

\begin{abstract}
We consider an optimization problem related to elliptic PDEs of the form $-\dive(a(x)\nabla u)=f$ with Dirichlet boundary condition on a given domain $\Om$. The coefficient $a(x)$ has to be determined, in a suitable given class of admissible choices, in order to optimize a given criterion. We first deal with the case when the cost is the so-called elastic compliance, and then we discuss the more general case when the problem is written as an optimal control problem.
\end{abstract}


\medskip


\section{Introduction}\label{intro}

Various quantities are involved in the study of elliptic PDEs, which we often refer to as {\it data}. In particular, in several situations the coefficients of an elliptic operator have to be determined in order to optimize a given cost functional. In all the paper $\Om$ is a given bounded open subset of $\RR^d$ and $H^1_0(\Om)$ is the usual Sobolev space of functions with zero boundary trace.

We give here below a presentation of the optimization problem, and in Section \ref{snum} we provide some numerical simulation. Other kinds of optimization problems for elliptic PDEs, namely the optimal choice of lower order terms, together with their regularity, have been considered in \cite{BCM}, \cite{Mon}, \cite{Sha}.

\subsection{Position of the problem}\label{ss1}
We consider the problem of minimizing a cost functional of the form
$$J(u)=\int_\Om j(x,u)\,dx,$$
where $j(x,s)$ is a suitable cost integrand with the appropriate growth conditions, and $u$ is the solution of the elliptic equation
\beq\label{pde}
\begin{cases}
-\dive\big(a(x)\nabla u\big)=f&\hbox{in }\Om\\
u\in H^1_0(\Om).
\end{cases}\eeq
Here the right-hand side $f\in L^2(\Om)$ is prescribed, while the coefficient $a$ has to be chosen in a suitable admissible class $\A$ in order to minimize the functional $J$ above. The problem is then an optimal control problem, where $u$ is the state variable, $a$ the control variable, $J$ the cost functional, and \eqref{pde} is the state equation. This amounts then to the problem
$$\min\big\{J(u)\ :\ u\hbox{ solves \eqref{pde}, }a\in\A\big\}.$$
The admissible class $\A$ is usually given in the form
$$\A=\Big\{a\ge0\ :\ \int_\Om\psi(a)\,dx\leq 1\Big\}.$$\par
A simpler way to impose the constraint on $a$ is to write the problem in the form
\beq\label{pbor}\min_{a\ge0}\min_{u\in H^1_0(\Om)}\Big\{\int_\Om\big(j(x,u)+\lambda\psi(a)\big)\,dx\ :\ u\hbox{ solves \eqref{pde}}\Big\}.\eeq
where $\lambda>0$ plays the role of a Lagrange multiplier. Moreover, we will assume $\psi$ convex and non-negative. Replacing $\psi$ by $\lambda\psi$ we can also assume $\lambda=1$.

\section{The compliance and energy problems.}\label{secPar}
A particular case of the problem considered above occurs when the cost $J$ is the so-called {\it compliance}, that is
$$j(x,u)= f(x)u.$$
In this case an easy integration by parts transforms the problem in the max/min problem
$$\max_{a\ge0}\min_{u\in H^1_0(\Om)}\int_\Om\Big(\frac12 a(x)|\nabla u|^2-f(x)u-\frac12\psi(a)\Big)\,dx.$$\par
We first assume that $\psi$ is superlinear, that is
\beq\label{superl}
\lim_{s\to+\infty}\frac{\psi(s)}{s}=+\infty,
\eeq
and we set
$$\E(a)=\inf_{u\in C^1_0(\Om)}\int_\Om\Big(\frac12 a(x)|\nabla u|^2-f(x)u-\frac12\psi(a)\Big)\,dx.$$
Concerning the right-hand side $f$, in several problems some concentration phenomena for data occur, so we simply require that the right-hand side $f$ is a signed measure.

\begin{Thm}\label{exL1}
Under assumption \eqref{superl}, the functional $\E(a)$ admits a maximizer $a_{opt}\in L^1(\Om)$, provided the right-hand side $f$ is such that $\E(a)>-\infty$ for at least a coefficient $a\in L^1(\Om)$.
\end{Thm}

\begin{proof}
Since for every $u\in C^1_0(\Om)$ the map
$$a\mapsto\int_\Om\Big(\frac12 a(x)|\nabla u|^2-\frac12\psi(a)\Big)\,dx-\int u\,df$$
is weakly $L^1(\Om)$ upper semicontinuous, the functional $\E(a)$ is weakly $L^1(\Om)$ upper semicontinuous, being the infimum of a family of upper semicontinuous functions. In addition, testing with $u=0$, we have
$$\E(a)\le-\frac12\int_\Om\psi(a)\,dx.$$
Hence, by the superlinearity of $\psi$ and by the well-known weak $L^1(\Om)$ compactness theorem, the existence of an optimal coefficient $a_{opt}$ is easily established by means of the direct methods of the calculus of variations.
\end{proof}

The case when $\psi$ has only a linear growth:
\beq\label{ling}
\lim_{s\to+\infty}\frac{\psi(s)}{s}=k>0
\eeq
is more delicate. In fact, in this case a more careful definition of the integrals
$$\int_\Om|\nabla u|^2\,da\qquad\hbox{and}\qquad\int_\Om\psi(a)$$
is needed. We refer to \cite{BGL} for more details about this case, which has strong links with the theory of optimal transportation, as first shown in \cite {BBS} and \cite{BB}. However, by an argument similar to the one of Theorem \ref{exL1}, an optimal coefficient $a_{opt}$ still exists, but in the larger class $\M(\Om)$ of nonnegative measures on $\Om$, as stated below.

\begin{Thm}\label{exM}
Under assumption \eqref{ling}, the functional $\E(a)$ admits a maximizer $a_{opt}$ in the class $\M(\Om)$, provided the right-hand side $f$ is such that $\E(a)>-\infty$ for at least a coefficient $a\in\M(\Om)$.
\end{Thm}

It is interesting to characterize the optimal coefficient $a_{opt}$ in terms of some suitable auxiliary variational problem. Thanks to a well-known result from min/max theory that allows to exchange the order of inf and sup, due to the convexity with respect to the variable $u$ and the concavity with respect to the variable $a$ (see for instance \cite{Cl} and \cite{Ek}), the initial problem becomes
\beq\label{formco}\inf_{u\in C^1_0(\Om)}\sup_{a\ge0}\int_\Om\Big(\frac12 a(x)|\nabla u|^2-\frac12\psi(a)\Big)\,dx-\int u\,df.\eeq
The supremum with respect to $a$ can be now easily computed:
$$\sup_{a\ge0}\int_\Om\Big(\frac12 a(x)|\nabla u|^2-\frac12\psi(a)\Big)\,dx-\int u\,df=\int_\Om\frac12 \psi^*\big(|\nabla u|^2\big)\,dx-\int u\,df,$$
where $\psi^*$ is the Legendre-Fenchel conjugate function of $\psi$. The auxiliary variational problem is then
\beq\label{foruCo}
\inf_{u\in C^1_0(\Om)}\int_\Om\frac12 \psi^*\big(|\nabla u|^2\big)\,dx-\int u\,df.
\eeq
Since $\psi^*(t)\ge t-\psi(1)$ it is easy to see that, at least when $f\in H^{-1}(\Om)$, the auxiliary variational problem admits a solution $\bar u\in H^1_0(\Om)$. Moreover, if $\psi$ is strictly increasing, then the function $s\mapsto\psi^*(s^2)$ is strictly convex and therefore $\bar u$ is unique. The optimal coefficient $a_{opt}$ can now be recovered through the optimality condition
\beq\label{conopC} a_{opt}|\nabla\bar u|^2=\psi(a_{opt})+\psi^*(|\nabla\bar u|^2).\eeq

\begin{Rem}\label{measures}
In some situations it is important to allow the right-hand side $f$ to be singular, for instance with concentrations on regions of lower dimensions. In general we can assume that $f\in\M$, the class of measures with finite mass; even if for some choice of the coefficient $a$ we may have $\E(a)=-\infty$, the optimal compliance problem is still meaningful, because these ``bad'' coefficients are ruled out by the optimization criterion, consisting in maximizing $\E(a)$. Moreover, all the arguments also apply to the case when, instead of having a boundary Dirichlet condition, we have the Neumann one, assuming as usual that the right-hand side $f$ has zero average.
\end{Rem}

\begin{Rem}\label{regul}
When $\psi(s)=\gamma s$, using the equivalence between coefficient optimization and optimal transport problem, pointed out in \cite{BB}, the following summability properties for the optimal coefficient $a_{opt}$ have been obtained (see \cite{PEP}):
\[\begin{split}
&f\in\M\quad\implies\quad\mu_{opt}\in\M\text{ possibly not unique;}\\
&f\in L^1(\Om)\quad\implies\quad\mu_{opt}\in L^1(\Om)\text{ and is unique;}\\
&f\in L^p(\Om)\quad\implies\quad\mu_{opt}\in L^p(\Om)\text{ for every }p\in[1,+\infty];\\
&\spt(\mu_{opt})\subset\text{convex envelope of }
\begin{cases}
\spt(f)&\text{in the Neumann case}\\
\spt(f)\cup\partial\Om&\text{in the Dirichlet case.}
\end{cases}
\end{split}\]
In addition, a mild $BV$ and $W^{1,1}$ regularity for $\mu_{opt}$ is available in some cases in dimension two. More precisely, when $d=2$ and under some additional assumptions on the regularity of $\Om$ and on the behavior of the datum $f$, we have (see \cite{Dw}):
\[\begin{split}
&f\in BV(\Om)\cap L^\infty(\Om)\quad\implies\quad\mu_{opt}\in BV(\Om),\\
&f\in W^{1,1}(\Om)\cap L^\infty(\Om)\quad\implies\quad\mu_{opt}\in W^{1,1}(\Om).
\end{split}\]
In higher dimension. Furthermore, the correspondence between the optimization and transport problems is unclear when the function $\psi$ is nonlinear.
\end{Rem}

\begin{Exa}\label{Ex1}
Taking $\psi(s)=s^2/2$ and $f\in W^{-1,4/3}(\Om)$, we obtain the auxiliary variational problem
$$\min\Big\{\int_\Om\Big(\frac14|\nabla u|^4-f(x)u\Big)\,dx\ :\ u\in W^{1,4}_0(\Om)\Big\},$$
or equivalently the nonlinear PDE
$$-\Delta_4u=f,\qquad u\in W^{1,4}_0(\Om),$$
whose unique solution $\bar u$ provides the optimal coefficient $a_{opt}(x)=|\nabla\bar u(x)|^2$. For instance, if $\Om$ is the unit ball, and $f=1$ we obtain
$$\bar u(x)=\frac{3}{4d^{1/3}}\big(1-|x|^{4/3}\big),\qquad a_{opt}(x)=\frac{|x|^{2/3}}{d^{2/3}}.$$
Conversely, taking $\Om$ the unit disc in $\RR^2$ and $f=\delta_0$ the unit Dirac mass at the origin, gives
$$\bar u(x)=\frac{3}{(16\pi)^{1/3}}(1-|x|^{2/3}),\qquad a_{opt}(x)=(2\pi|x|)^{-2/3}.$$
\end{Exa}

\begin{Exa}\label{Ex2}
Taking
\beq\label{psco}\psi(s)=\begin{cases}
\gamma s&\hbox{if }\alpha\le s\le\beta\\
+\infty&\hbox{otherwise,}
\end{cases}\eeq
with $0<\alpha<\beta$, $\gamma>0$, we have the auxiliary variational problem
$$\min\Big\{\int_\Om\frac{|\nabla u|^2-\gamma}{2}\Big(\beta1_{|\nabla u|^2\ge\gamma}+\alpha1_{|\nabla u|^2\le\gamma}\Big)-f(x)u\,dx\ :\ u\in H^1_0(\Om)\Big\},$$
whose unique solution $\bar u$ provides the optimal coefficient $a_{opt}\in L^\infty(\Om)$. It has been proved in \cite{Cas} (see also \cite{Cas2}) that, when $\Om$ is of class $C^{1,1}$ and $f\in L^2(\Om)$, then $\bar u$ is in $H^2(\Om)$ and $\nabla a_{opt}\cdot\nabla \bar u$ belongs to $L^2(\Om)$.
\end{Exa}

Another case where \eqref{pbor} reduces to a variational problem is the minimization of the energy, corresponding to
$$j(x,u)=-f(x)u.$$
Similarly to (\ref{formco}), the problem becomes
\beq\label{pbcsr}
\min_{a\ge0}\min_{u\in H^1_0(\Om)}\int_\Om\Big(\frac12 a(x)|\nabla u|^2-f(x)u+\frac12\psi(a)\Big)\,dx.
\eeq
which, computing the minimum in $a$, can be written as
\beq\label{pbume}\min_{u\in H^1_0(\Om)}\int_\Om\Big(-{1\over 2}\psi^\ast\big(-|\nabla u|^2\big)-fu\Big)\,dx.\eeq

Imposing that
$$\lim_{s\to 0^+}s\psi(s)=\infty,$$
the functional becomes coercive over $W^{1,1}_0(\Om)$. Assuming also that $\psi$ is decreasing and that $\psi(1/s)$ is convex, we have that the functional in (\ref{pbume}) is convex. Therefore, under these assumptions and taking $f$ smooth enough, problem (\ref{pbume}) has a solution in $W^{1,1}_0(\Om)$ (it is not necessarily in $H^1_0(\Om)$). However, in other situations this functional is not convex and then (\ref{pbume}) may not have a solution. To avoid this difficulty it is necessary to deal with a relaxed problem formulation consisting in replacing the function $\xi\in\RR^d\mapsto -\psi(-|\xi|^2)$
by its convex hull. 

\begin{Exa}\label{Ex3}
Related to Example \ref{Ex1}, we take $\psi(s)=1/(2s^2)$. Then problem (\ref{pbume}) becomes
$$\min_{u\in W^{1,{4\over 3}}_0(\Om)}\int_\Om\Big({3\over 4}|\nabla u|^{4\over 3}-fu\Big)\,dx,$$
which has a unique solution $\overline u$ if $f$ is in $W^{-1,4}(\Om)$. The optimal control is given by
$a_{opt}=|\nabla \overline u|^{-2/3}$. In this way, if $\Om$ is the unit ball in $\RR^d$ and $f=1$, we get
$$\overline u(x)={1-|x|^4\over 4d^3},\qquad a_{opt}(x)={d^2\over |x|^2}.$$
\end{Exa}
\begin{Exa}\label{Ex4} 
Taking
\beq\label{envconv}
\psi(s)=\begin{cases}
\gamma(\beta-s)&\text{if }s\in[\alpha,\beta]\\
+\infty&\text{otherwise,}
\end{cases}\eeq
 with $0<\alpha<\beta$, $\gamma>0$, we have 
 $$-\psi^\ast(-|\xi|^2)=\begin{cases}
\beta|\xi|^2&\text{if }|\xi|^2\le\gamma\\
\alpha|\xi|^2+\gamma(\beta-\alpha)&\text{if }|\xi|^2>\gamma,
\end{cases}$$
which is not convex. Computing its convex hull (see e.g. \cite{GoKoRe}) we get the relaxed formulation
\beq\label{pbume2}\min_{u\in H^1_0(\Om)}\int_\Om\Big({1\over 2}\phi\big(|\nabla u|\big)-fu\Big)\,dx,\eeq
with
$$\phi(s)=\begin{cases}
\beta s^2&\text{if }\dis s^2\le{\alpha\over\beta}\gamma\\
2\sqrt{\alpha\beta\gamma}\,s-\alpha\gamma&\text{if }\dis{\alpha\over\beta}\gamma\le s^2\le{\beta\over\alpha}\gamma\\
\alpha s^2+\gamma(\beta-\alpha)&\text{if }\dis{\beta\over\alpha}\gamma\le s^2.
\end{cases}$$
The Euler-Lagrange equation for \eqref{pbume2} proves that for a given solution $\overline u$, the associated optimal \lq\lq relaxed control\rq\rq\space 
 is given by
$$a_{opt}={\phi'\big(|\nabla \overline u|\big)\over 2|\nabla\overline u|}.$$
It can be proved that this optimal relaxed control coincides with the optimal relaxed control defined in the following section. Moreover, although $a_{opt}$ and $\nabla \overline u$ may not be unique, the function $\overline\sigma=a_{opt}\nabla\overline u$ is unique.\par
Taking $\Om$ the unit ball in $\RR^d$, $f=1$, $\gamma<1/(d^2\alpha\beta)$, and denoting $\tau=d\sqrt{\alpha\beta\gamma}$, we have
$$\overline u(x)=\left\{\ba{ll}\dis {(1-\tau^2)(\beta-\alpha)\over 2d\alpha\beta}+{1-|x|^2\over 2d\beta} &\hbox{ if }|x|<\tau\\ \ecart\dis
{1-|x|^2\over 2d\alpha} &\hbox{ if }|x|>\tau,\ea\right.\qquad a_{opt}(x)=\left\{\ba{ll}\dis \beta &\hbox{ if }|x|<\tau\\ \ecart\dis
\alpha &\hbox{ if }|x|>\tau.\ea\right.$$
\end{Exa}

\begin{Rem}\label{regul}
For $\psi$ given by (\ref{envconv}), it has been proved in \cite{Cas} that $\Om\in C^{1,1}$ and $f\in L^2(\Om)$ imply that $\overline u$ is in $H^2(\Om)^d$ and that the derivatives of $a_{opt}$ in the orthogonal directions to $\sigma$ are in $L^2(\Om)$.
\end{Rem}

\section{The general problem.}

In the case of a general optimal control problem of the form
\beq\label{pbge}
\min_{a\ge0}\min_{u\in H^1_0(\Om)}\Big\{\int_\Om\big(j(x,u)+\psi(a)\big)\,dx\ :\ u\hbox{ solves \eqref{pde}}\Big\},
\eeq
the existence of an optimal coefficient $a_{opt}$ may fail, and a solution exists only in a {\it relaxed} sense. For $0<\alpha<\beta$, a counterexample to the existence of a solution $a_{opt}$ can be found in \cite{Mur}, where
$$\psi(s)=\begin{cases}
0&\hbox{if }\alpha\le s\le\beta\\
+\infty&\hbox{otherwise,}
\end{cases}$$
and
$$j(x,s)=|s-u_0(x)|^2$$
for a suitable function $u_0$. A different counterexample is illustrated in Section \ref{counter}.

In order to understand what relaxed solutions are, we have to recall the notion of $G$-convergence, introduced by De Giorgi and Spagnolo in \cite{DGS}: a sequence $a_n(x)$ of functions between $\alpha$ and $\beta$ is said to $G$-converge to a symmetric $d\times d$ matrix $A(x)$ if for every $f\in L^2(\Om)$ the solutions $u_n$ of the PDEs
$$-\dive\big(a_n\nabla u_n\big)=f,\qquad u_n\in H^1_0(\Om)$$
converge in $L^2(\Om)$ to the solution $u$ of the PDE
\beq\label{ecmat}-\dive\big(A\nabla u\big)=f,\qquad u\in H^1_0(\Om).\eeq
The question becomes now to characterize the $G$-closure $\overline\A$ of the set of coefficients $a_n$. A complete answer has been given by Murat and Tartar in \cite{MT}, \cite{Tar} (see also \cite{LC} for the two-dimensional case). They proved that the $G$-closure $\overline\A$ above consists of all symmetric $d\times d$ matrices $A(x)$ whose eigenvalues $\lambda_1(x)\le\lambda_2(x)\le\dots\le\lambda_d(x)$ are between $\alpha$ and $\beta$ and satisfy for a suitable $t\in[0,1]$ (depending on $x$) the following $d+2$ inequalities:
$$\begin{cases}
\dis\sum_{1\le i\le d}\frac{1}{\lambda_i-\alpha}\le\frac{1}{\nu_t-\alpha}+\frac{d-1}{\mu_t-\alpha}\\
\dis\sum_{1\le i\le d}\frac{1}{\beta-\lambda_i}\le\frac{1}{\beta-\nu_t}+\frac{d-1}{\beta-\mu_t}\\
\nu_t\le\lambda_i\le\mu_t\qquad i=1,\dots,d,
\end{cases}$$
being $\mu_t$ and $\nu_t$ respectively the arithmetic and the harmonic mean of $\alpha$ and $\beta$, namely
$$\mu_t=t\alpha+(1-t)\beta,\qquad\nu_t=\Big(\frac{t}{\alpha}+\frac{1-t}{\beta}\Big)^{-1}.$$
For instance, when $d=2$, the set above is given by the symmetric $2\times2$ matrices $A(x)$ whose eigenvalues $\lambda_1(x)$ and $\lambda_2(x)$ are between $\alpha$ and $\beta$ and
$$\frac{\alpha\beta}{\alpha+\beta-\lambda_1(x)}\le\lambda_2(x)\le\alpha+\beta-\frac{\alpha\beta}{\lambda_1(x)}\;.$$

\begin{figure}[h!]
\centering
{\includegraphics[scale=1.0]{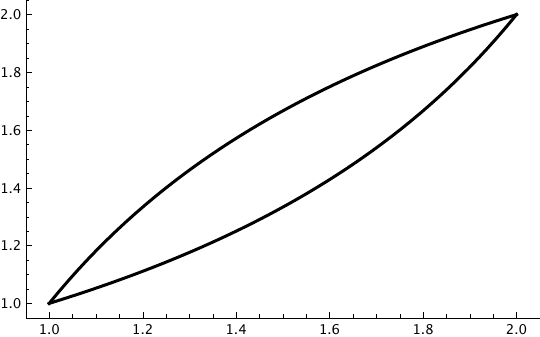}}
\caption{Attainable matrices, in the plane $(\lambda_1,\lambda_2)$, for $d=2$, $\alpha=1$, $\beta=2$.}
\end{figure}

As far as we know, an explicit form of the relaxation
$$\Psi(A)=\inf_{a_n\to_G A}\liminf_n\int_\Om\psi(x,a_n)\,dx;$$
with a general function $\psi(x,a)$ is not known. The case $\psi(x,a)=g(x)a$ has been considered in \cite{Cab} and \cite{CDM}. For instance, if $\psi$ is given by \eqref{psco}, denoting by $\overline\A$ the $G$-closure described above, the relaxation $\Psi(A)$ is given by
$$\Psi(A)=\begin{cases}\dis
\int_\Om\gamma\lambda_{max}\big(A(x)\big)\ dx &\text{if }A\in\overline\A\\
+\infty&\text{otherwise,}
\end{cases}$$
being $\lambda_{max}(A)$ the largest eigenvalue of the $d\times d$ symmetric matrix $A$. Namely, taking into account that the solution of the state equation 
$$-\dive\big(A(x)\nabla u\big)=f,\qquad u\in H^1_0(\Om),$$
does not vary if we replace $A(x)$ by another matrix function $B(x)$ such that
$$A(x)\nabla u=B(x)\nabla u\qquad\text{a.e. in }\Om,$$
and that for every $\xi\in\RR^d$ one has
$$\big\{A\xi: \ A\in \overline\A\big\}= \big\{\eta\in\RR^d: \ (\eta-\nu_t\xi)\cdot(\eta-\mu_t\xi)\leq 0\big\},$$
we have that a relaxation of (\ref{pbge}) with $\psi(x,a)=g(x)a$ is given by
\beq\label{pbrege}\min_{\nu_tI\leq A\leq \mu_tI\atop 0\leq t\leq 1} \min_{u\in H^1_0(\Om)}\Big\{\int_\Om\big(j(x,u)+g(x)\mu_t\big)\,dx\ :\ u\hbox{ solves \eqref{ecmat}}\Big\},\eeq
where $I$ denotes the identity matrix.

As an example, we can consider the energy problem $j(x,u)=-f(x)u$ with $\psi$ given by (\ref{envconv}). Then, taking into account (\ref{pbcsr}) and (\ref{pbrege}) the relaxed problem can be written as
\beq\label{pbregeE}\min_{ 0\leq t\leq 1} \min_{u\in H^1_0(\Om)}\Big\{\int_\Om\big({\nu_t\over 2}|\nabla u|^2-f(x)u-\gamma\mu_t\big)\,dx\Big\}.\eeq
Using the minimum in $t$ this proves again that $u$ solves (\ref{pbume2}).\par

Thanks to (\ref{pbrege}) we obtain a system of optimality conditions. For this purpose, we assume the function $j(x,s)$ derivable with respect to $s$ with appropriate growth conditions.

Take $(t_{opt},A_{opt})$ an optimal solution of the relaxed problem, then for any admissible control $(t,A)$ and $0\le\ep\le1$ the control
$$\big(t_{opt}+\ep(t-t_{opt}), A_{opt}+\ep (A-A_{opt})\big)$$
is also admissible. Using it and deriving with respect to $\ep$ we conclude that
\beq\label{proma}A_{opt}\nabla \bar u\cdot\nabla \bar p+g(\beta-\alpha)t_{opt}=\max_{\nu_tI\leq A\leq \mu_tI\atop 0\leq t\leq 1}\Big\{A_{opt}\nabla \bar u\cdot\nabla\bar p-g(\beta-\alpha)t_{opt}\Big\},\eeq
with $\bar u,\bar p$ the state and adjoint state functions, solutions of
\beq\label{deesda}
-\dive\big(A_{opt}\nabla \bar u\big)=f,\quad -\dive\big(A_{opt}\nabla \bar p\big)=\partial_sj(x,\bar u),\qquad \bar u,\bar p\in H^1_0(\Om).
\eeq
Computing the maximum in \eqref{proma}, we obtain the optimality conditions (see for instance \cite{All}, \cite{MT})
\beq\label{conopg1}
\begin{cases}
\dis A_{opt}\nabla\bar u={\mu_{t_{opt}}+\nu_{t_{opt}}\over 2}\nabla\bar u+
{\mu_{t_{opt}}-\nu_{t_{opt}}\over2}{|\nabla\bar u|\over|\nabla\bar p|}\nabla \bar p&\text{a.e. in }\{\nabla\bar p\not=0\}\\
\dis A_{opt}\nabla\bar p={\mu_{t_{opt}}+\nu_{t_{opt}}\over 2}\nabla\bar p+
{\mu_{t_{opt}}-\nu_{t_{opt}}\over 2}{|\nabla\bar p|\over |\nabla\bar u|}\nabla \bar u&\text{a.e. in }\{\nabla\bar u\not=0\},
\end{cases}\eeq
\beq\label{conopg2}t_{opt}=
\begin{cases}
0&\dis\text{if }g<N^+-{\beta\over\alpha}N^-\\
\dis{1\over\beta-\alpha}\Big(\sqrt{\alpha\beta N^-\over N^+-g}-\alpha\Big)&\dis\text{if }N^+-{\beta\over\alpha}N^-\le g\le N^+-{\alpha\over\beta}N^-\\
1&\dis\text{if }N^+-{\alpha\over\beta}N^-< g,
\end{cases}\eeq
with
\beq\label{N+N-}
N^+={|\nabla u||\nabla p|+\nabla u\cdot\nabla p\over 2},\qquad N^-={|\nabla u||\nabla p|-\nabla u\cdot\nabla p\over 2}.
\eeq

\section{Nonexistence of an optimal coefficient}\label{counter}

In this section we provide a counterexample to the existence of an optimal coefficient $a_{opt}$ for \eqref{pbor}. We take $\Om=B(0,1)$ the unit ball in $\RR^d$, and
\beq\label{Ex4}
\begin{cases}
f=1\text{ the right-hand side,}\\
j(x_1,...,x_d,s)=(1+\ep x_1)s\text{ with }\ep>0,\\
\dis\psi(s)=\begin{cases}
\tau^2 s&\text{if }s\in[1,2]\\
+\infty&\text{otherwise}\end{cases}
\qquad\hbox{ with }0<\tau<{1\over d}.
\end{cases}\eeq
Let us prove the existence of $\ep_0>0$ such that \eqref{pbor} has no solution for $0<\ep<\ep_0$.

First we observe that for $\ep=0$ problem (\ref{pbor}) is a particular case of the compliance problem considered in Section \ref{secPar}. By \eqref{foruCo} we get that the state function $u_0$ associated to an optimal control $a_0$ is the unique solution of \eqref{foruCo} with
$$\psi^\ast(s)=\begin{cases}
s-\tau^2&\text{if }s<\tau^2\\
2(s-\tau^2)&\text{if }s>\tau^2.
\end{cases}$$
By uniqueness $u_0$ is invariant by rotations and then is a radial function $u_0(r)$. Combined with \eqref{conopC}, this implies
\beq\label{up0a0} u'_0(r)=\begin{cases}
\dis-{r\over d}&\text{if }r<d\tau\\
-\tau&\dis\text{if }d\tau\le r\le2d\tau\\
\dis-{r\over2d}&\text{if }2\tau d<r,\end{cases}
\qquad a_0(r)=\begin{cases}
1&\text{if }r<d\tau\\
\dis{r\over d\tau}&\text{if }d\tau<r<2d\tau\\
2&\text{if }r>2d\tau.
\end{cases}\eeq

On the other hand, recalling that the solution $u$ of the state equation in \eqref{pde} satisfies $a\nabla u=\sigma$, with $\sigma$ the solution of
$$\min\Big\{\into {|\zeta|^2\over a}\,dx\ :\ -\dive\zeta=1\Big\},$$
and that
$$\into u\,dx=\into{|\sigma|^2\over a}\,dx,$$
we get that \eqref{pbor} with $\ep=0$ is also equivalent to
$$\min_{1\leq a\leq 2}\min_{-\dive\sigma=1}\into \Big({|\sigma|^2\over a}+\tau^2a\Big)\,dx.$$
Taking the minimum with respect to $a$, this provides
$$a(x)=\begin{cases}
1&\text{if }|\sigma(x)|<\tau\\
\dis{|\sigma(x)|\over\tau}&\text{if }\tau\le|\sigma(x)|\le2\tau\\
2&\text{if }|\sigma(x)|>2\tau,
\end{cases}$$
with $\sigma$ a solution of
\beq\label{pbmisi}
\min_{-\dive\sigma=1}\into\Upsilon(|\sigma|)\,dx,\qquad\Upsilon(s)=
\begin{cases}
s^2+\tau^2&\text{if }0\le s<\tau\\
2\tau s&\text{if }\tau\le s\le2\tau\\
\dis{s^2\over2}+2\tau^2&\text{if }s>2\tau.
\end{cases}\eeq
By what proved above, this problem has a unique solution $\sigma_0 (x)=-x/d$.

Let us now prove the non-existence result. Arguing by contradiction, we assume there exist $\ep_k>0$ tending to zero and $a_{\ep_k}$ solution of \eqref{pbor}. We denote by $u_{\ep_k}$ the corresponding state function. Since the state equation does not depend on $\ep$, we have
\beq\label{desueu0}
\into\big((1+\ep_k x_1)u_{\ep_k}+\tau^2a_{\ep_k}\big)\,dx\le\into \big((1+\ep_k x_1)u_0+\tau^2a_0\big)\,dx.
\eeq 
Using also
$$\into\big(u_0+\tau^2a_0)\,dx=\into\Upsilon(|\sigma_0|)\,dx,$$
and that $\sigma_0$ solves \eqref{pbmisi}, we deduce from \eqref{desueu0}
\beq\label{despue}
0\le\into\big(\Upsilon(|\sigma_\ep|)-\Upsilon(|\sigma_0|)\big)dx\le\ep\into x_1(u_0-u_\ep)\,dx.
\eeq

Taking into account that $u_{\ep_k}$ solves \eqref{pde} with $a=a_{\ep_k}$, we know (\cite{DGS}) that there exist a subsequence, still denoted by $\ep_k$, and $A\in L^\infty(\Om)^{d\times d}$ symmetric such that
$$I\le A\le2I\text{ a.e. in }\Om,\qquad a_{\ep_k}I\stackrel{G}\rightharpoonup A,$$
$$u_{\ep_k}\rightharpoonup u\text{ in }H^1_0(\Om),\qquad\sigma_{\ep_k}\rightharpoonup A\nabla u\text{ in }L^2(\Om)^d,\qquad-\dive(A\nabla u)=1\text{ in }\Om.$$
Coming back to \eqref{despue} and taking into account the convexity of $\Upsilon$, we have that
$$\into\Upsilon(|A\nabla u|)\,dx=\min_{-\dive\zeta=1}\into\Upsilon(|\zeta|)\,dx.$$
Thus, $A\nabla u$ is a solution of \eqref{pbmisi} and then $u$ is a solution of \eqref{pbor}. By uniqueness this proves that
$$u=u_0,\qquad A=a_0I,\qquad\sigma=\sigma_0.$$
Now, we consider $p_\ep$ the adjoint state defined as the solution of
$$\begin{cases}
-\dive(a_{\ep_k}\nabla p_{\ep_k})=1+\ep_k x_1&\text{in }\Om\\
p_{\ep_k}=0&\text{on }\partial\Om.
\end{cases}$$
Then $z_{\ep_k}=(p_{\ep_k}-u_{\ep_k})/ \ep_k$ solves
$$\begin{cases}
-\dive(a_{\ep_k}\nabla z_{\ep_k})= x_1&\text{in }\Om\\
z_{\ep_k}=0&\text{on }\partial\Om.
\end{cases}$$
By the $G$-convergence of $a_{\ep_k} I$, we have
$$z_{\ep_k}\rightharpoonup z\text{ in }H^1_0(\Om),\qquad a_{\ep_k}\nabla z_{\ep_k}\rightharpoonup A\nabla z\text{ in }L^2(\Om)^d,$$
with $z$ the solution of
$$\begin{cases}
-\dive(A\nabla z)=x_1&\text{in }\Om\\
z=0&\text{on }\partial\Om.
\end{cases}$$
Since we are assuming $a_{\ep_k}$ a solution of \eqref{pbor} and then of the relaxed problem \eqref{pbrege}, we deduce from \eqref{conopg1} that the matrix with columns $\nabla u_{\ep_k},\nabla p_{\ep_k}$ has rank one. By Morrey's theorem relative to the weak converges of the Jacobian (\cite{Mor}), which in our case reduces to
\[\begin{split}
&\partial_i u_{\ep_k}\partial_jz_{\ep_k}-\partial_iz_{\ep_k}\partial_j u_{\ep_k}=\partial_i(u_{\ep_k}\partial_jz_{\ep_k})-\partial_j(u_{\ep_k}\partial_iz_{\ep_k})\\
&\rightharpoonup\partial_i(u\partial_jz)-\partial_j(u_0\partial_iz)=\partial_i u\partial_jz-\partial_iz\partial_j u_0\quad\text{in }W^{-1,1}(\Om),
\end{split}\]
we have that $\nabla z$ is parallel a.e. to $\nabla u_0$, and that $z$ solves
$$\begin{cases}
-\dive\big(a_0\nabla z\big)=x_1&\text{in }\Om\\
z_0=0&\text{on }\partial\Om.
\end{cases}$$
However, the solution of this problems is given by $z(x)=h(|x|)x_1$ with $h$ the unique solution of
$$\begin{cases}
\dis-\big(a_0h'\big)'-{d+1\over r}a_0h'-{a'_0\over r}h=1&\text{in }(0,1)\\
\dis h(1)=0,\qquad\int_0^1r^{d+1}|h'|^2dr<\infty.
\end{cases}$$
Thus $z$ is not a radial function and $\nabla z$ is not parallel to $\nabla u_0$.
\cqfd

\section{Numerical simulations}\label{snum}
In this section we present some numerical experiments for the resolution of problems of the kind of (\ref{pbor}) in the 2-d case. We solve numerically the problems showed in Examples \ref{Ex1} and \ref{Ex2} for compliance optimization and the example given by problem (\ref{pbcsr}) with $\psi$ defined in (\ref{envconv}) for energy optimization. 

In the case of the compliance problem we put
$$J(a)=\int_\Om\big( f(x)u + \psi(a)\big)dx,$$
then, having in mind that for $u$ solution of the state equation (\ref{pde}), it is easy to compute that
$$\frac{d J(a)}{d a}\cdot\tilde a=\int_\Om\tilde a\left(\psi'(a)-|\nabla u|^2\right)\,dx.$$

We will apply a gradient descent method with projection into the appropriate subspace functions $a$ such that $\psi(a)<\infty$. For instance in Example \ref{Ex2} $\psi$ is finite where $a\in[\al,\be]$.
The algorithm is the following:
\begin{itemize}
\item Initialization: choose an admisible function $a_0$.
\item for $k\ge0$, iterate until convergence as follow:
\begin{itemize}
\item compute $u_k$ solution of (\ref{pde}) for $a=a_k$.
\item compute $\bar a_k=-( \psi'(a)- |\nabla u|^2)$ descent direction associated to $u_k$.
\item update the function $a_k$:
$$a_{k+1}=P_\psi(a_k+\epsilon_k \bar a_k)$$
where $P_\psi$ is a projection operator associated to the set $\{a: \psi(a)<\infty\}$, and where $\epsilon_k$ is small enough to ensure the decrease of the cost function.
\end{itemize}
\item Stop if convergence: $\frac{|J(a_{k})-J(a_{k-1)}|}{|J(a_0)|}<tol$, for $tol>0$ small. 
\end{itemize}
In the case of the energy problem, we use a similar algorithm based on formulation (\ref{pbregeE}).

On the other hand, we are interested also in showing the numerical evidence of the non-existence of an optimal coefficient for the general case. We propose to solve the relaxed formulation of \eqref{pbor}, given by \eqref{pbrege}, in $\Om=B(0,1)$ unit disc of $\RR^2$, and with the data given in (\ref{Ex4}). By convex minimization and following the system of optimality, we propose the following algorithm to compute $(t_{opt},A_{opt})$.

\begin{itemize}
\item Initialization: choose an admissible $(t_0,A_0)$ such that $ 0\leq t_0\leq 1$ and $\nu_tI\leq A_0\leq \mu_tI$.
\item For $k\ge0$, iterate until convergence as follows:
\begin{itemize}
\item compute the solutions $u_k$ and $p_k$ of (\ref{deesda}) for $A_{opt}=A_k$. Then, we define $N^+$, $N^-$ by (\ref{N+N-});
\item compute $\hat t$ given by (\ref{conopg2});
\item compute $\hat A$ defined by (\ref{conopg1}), considering $\hat u=u_k$, $\hat p=p_k$, $t_{opt}=\hat t$ and such that the spectrum of $\hat A$ belongs to $[\nu_{t_k}, \mu_{t_k}]$;
\item for $\ep_k\in(0,1]$, update the function $(t_k,A_k)$ as:
$$t_{k+1}=t_k+\ep_k(\hat t -t_k),\qquad A_{k+1}=A_k+\ep_k(\hat A -A_k).$$
\end{itemize}
\item Stop if convergence: $\frac{|J(t_{k+1},A_{k+1})-J(t_{k},A_{k})|}{|J(t_0,A_0)|}<tol$, for $tol>0$ small and $J$ corresponding with the cost function in (\ref{pbrege}) . 
\end{itemize}
We now show some numerical experiments based on the algorithms described above. The computation has been carried out using the free software FreeFem++ v4.5 (\cite{Hecht}, available in http://www.freefem.org). The picture of figures are made in Paraview 5.10.1 (available at https://www.kitware.com/open-source/\# paraview), which is free too. We use $P_1$-Lagrange finite element approximations for $u$ and $p$, solutions of the state and adjoint state equations respectively, and $P_0$-Lagrange finite element approximations for control variables, $a$ for scalar problems or $(t,A)$ for the matrix problems. For all simulations where the parameters $\al$ and $\be$ appear, we consider a normalized value $\al=1$, and $\be=2$. 

\begin{Exa}
We consider two cases in the framework of compliance optimization, i.e., $j(x,u)=f(x)u$. In Example \ref{Ex1} for $\psi(s)=s^2/2$ we provided an explicit solution when $\Om$ is the unit ball and $f\equiv1$. Here we solve numerically this problem in the non-radial case, considering the square $\Om=[0,1]^2$. In Figure \ref{Fig:Ex1} we show the computed optimal solutions, the optimal density $a$ on the left, and the optimal function $u$ on the right. 

\begin{figure}[!h]
\begin{minipage}[t]{7.4cm}
\centering
\includegraphics[width=1.\textwidth]{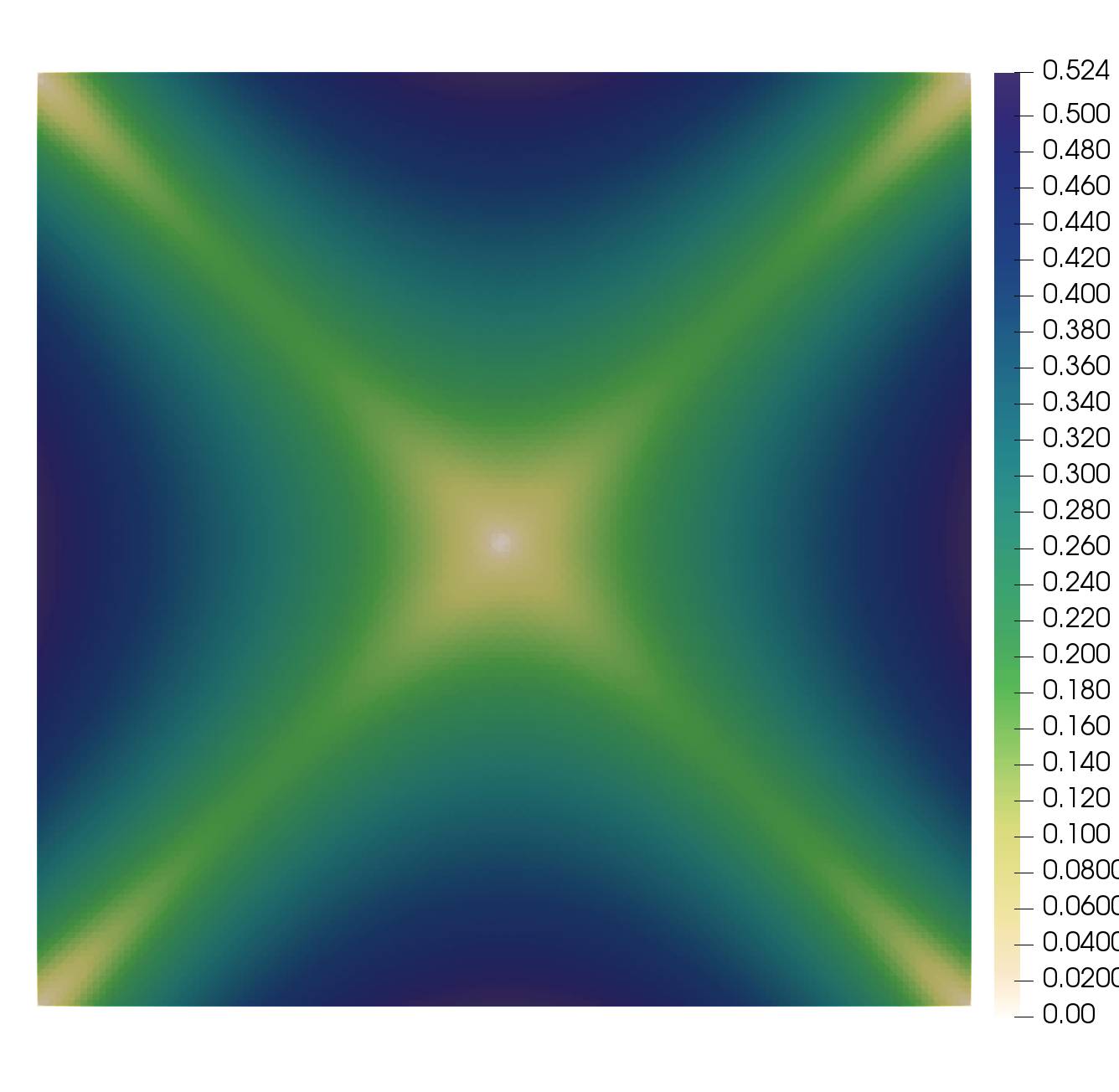}
\end{minipage}
\begin{minipage}[t]{7.4cm}
\centering
\includegraphics[width=1.\textwidth]{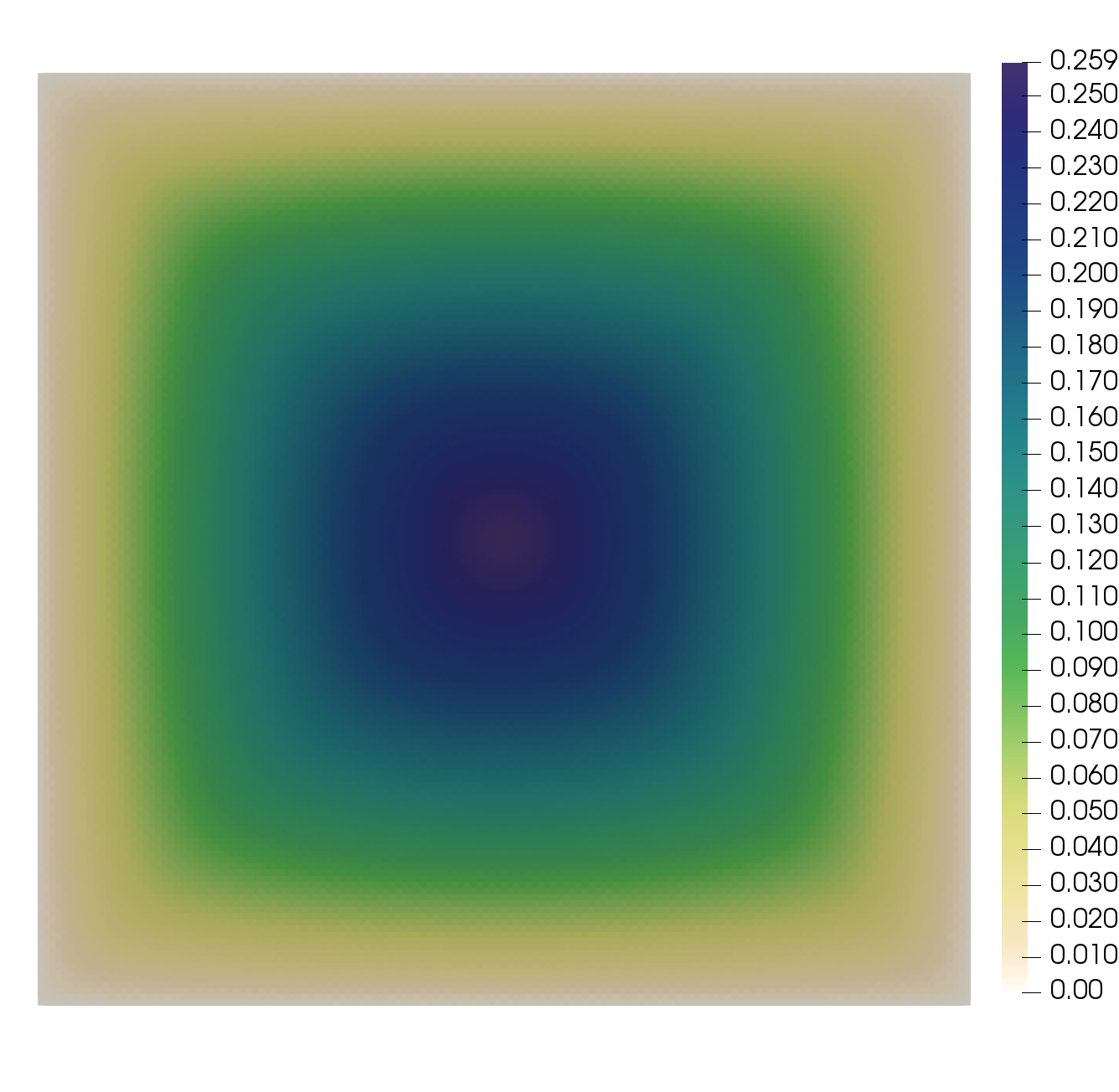}
\end{minipage}
\vskip-.5cm
\caption{Example 1.1 Optimal $a$ (left), and optimal $ u $ (right).}
\label{Fig:Ex1}
\end{figure}

The second case corresponds to $\psi$ given by \eqref{psco}. We solve numerically this problem in the cube $\Om=[0,1]^2$, and considering the Lagrange multiplier $\ga=0.01141$ in order to work with a volumen constraint of 50\% of each phase $\alpha$ and $\beta$. In Figure \ref{Fig:Ex2} we show the computed optimal solutions, the optimal density $a$ on the left, and optimal function $u$ on the right.

\begin{figure}[!h]
\begin{minipage}[t]{7.4cm}
\centering
\includegraphics[width=1.\textwidth]{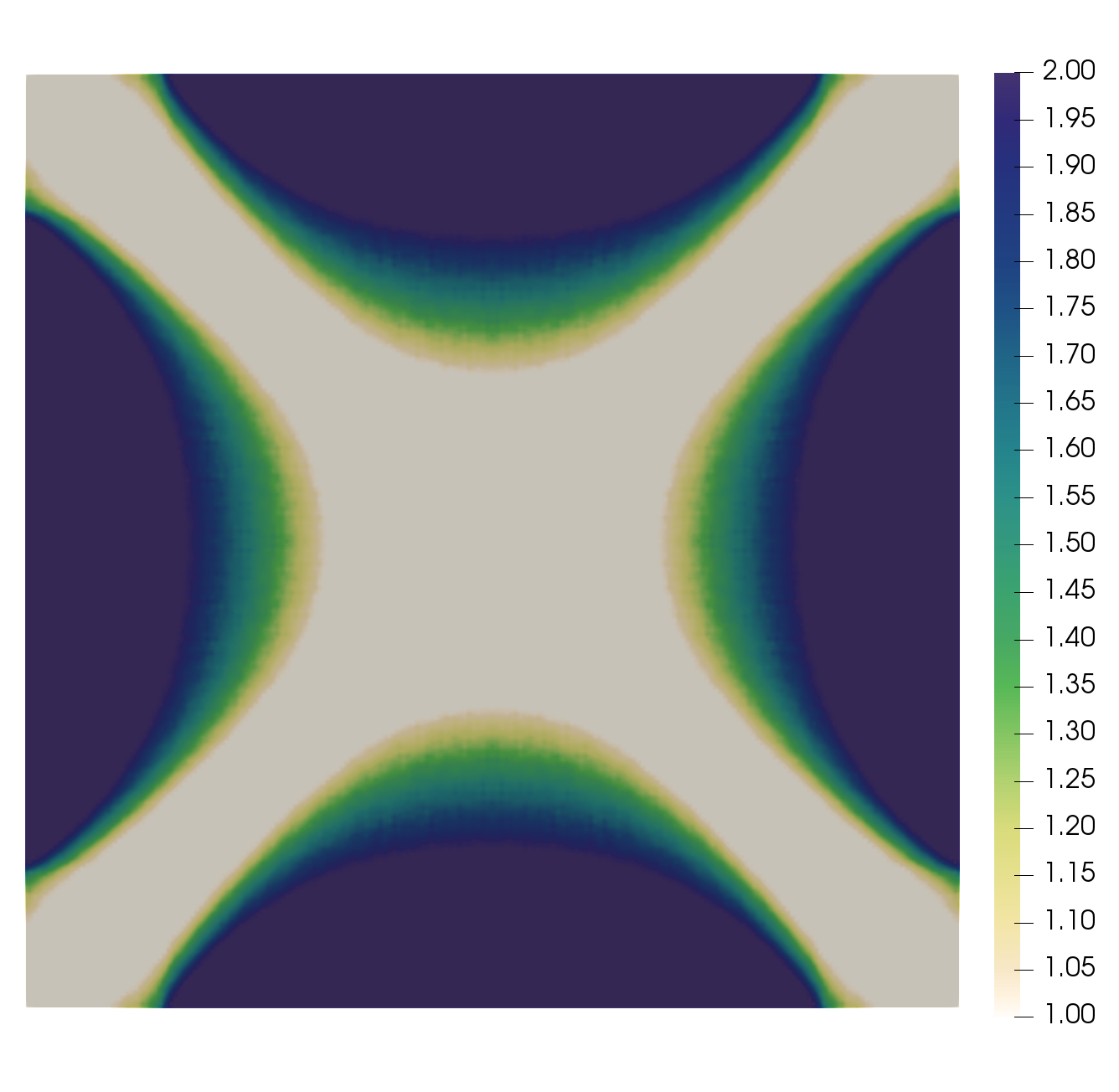}
\end{minipage}
\begin{minipage}[t]{7.4cm}
\centering
\includegraphics[width=1.\textwidth]{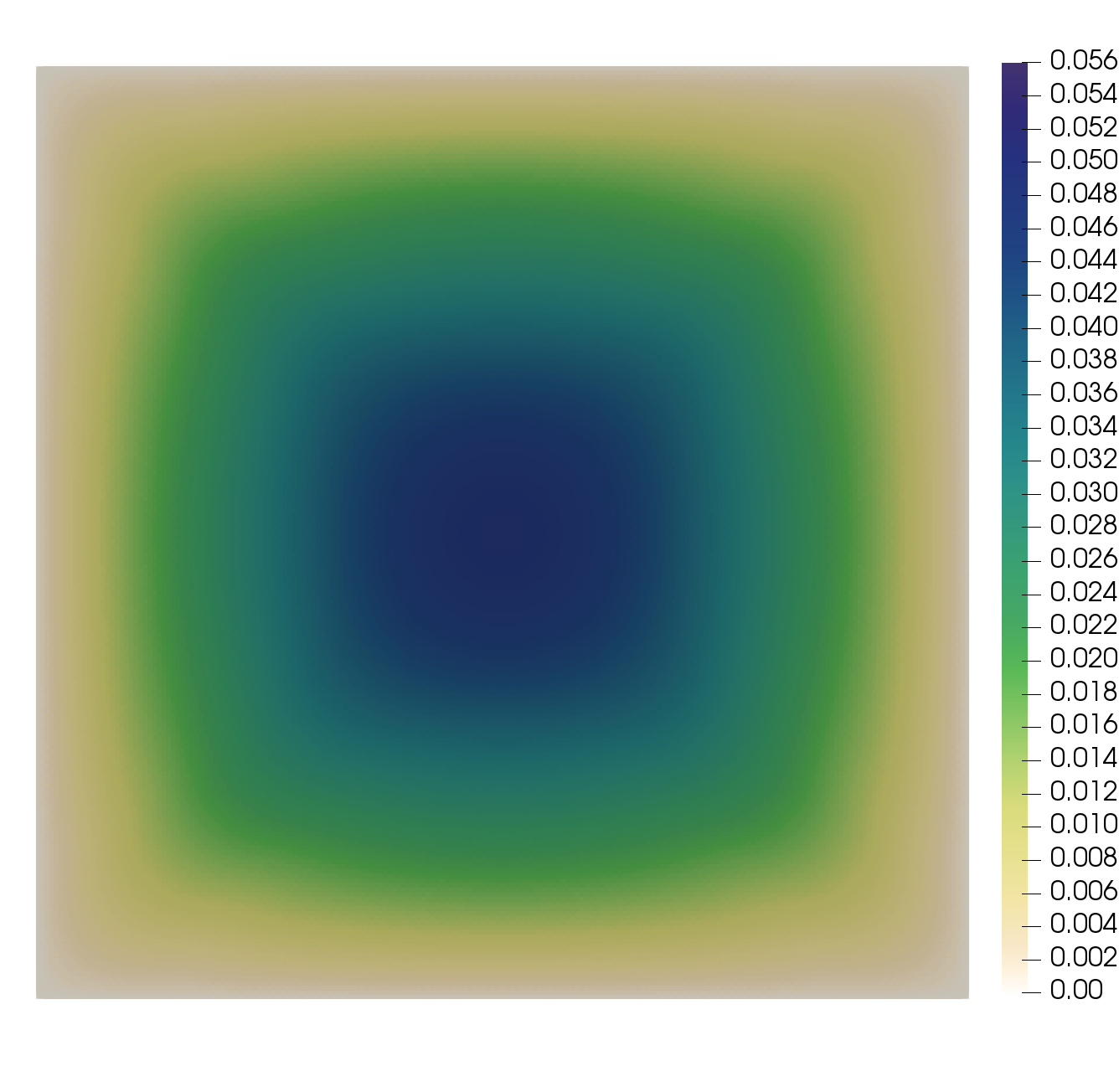}
\end{minipage}
\vskip-.5cm
\caption{Example 1.2 Optimal $a$ (left), and optimal $ u $ (right).}
\label{Fig:Ex2}
\end{figure}
\end{Exa}

\begin{Exa}
We consider a case in the framework of energy optimization, i.e., $j(x,u)= -f(x)u$. We assume $f\equiv1$ and $\psi$ defined as (\ref{envconv}) with $\ga=0.0142$ in order to assure a volumen constraint of 50\% of each phase $\alpha$ and $\beta$ and we deal with the relaxed formulation (\ref{pbregeE}). In Figure \ref{Fig:Ex3} we show the computed optimal solutions, the optimal density $a$ on the left, and optimal function $u$ on the right. 

\begin{figure}[!h]
\begin{minipage}[t]{7.4cm}
\centering
\includegraphics[width=1.\textwidth]{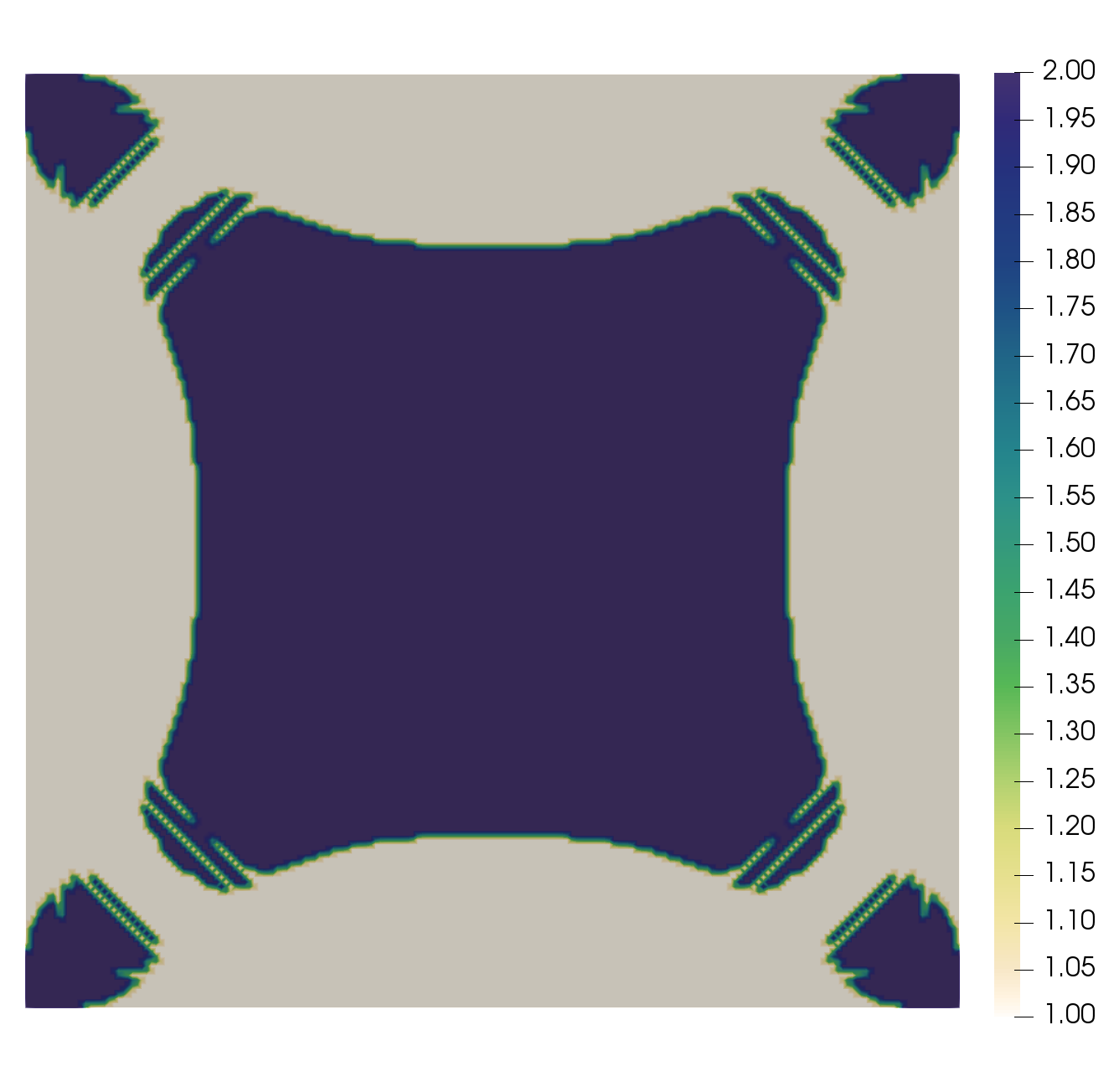}
\end{minipage}
\begin{minipage}[t]{7.4cm}
\centering
\includegraphics[width=1.\textwidth]{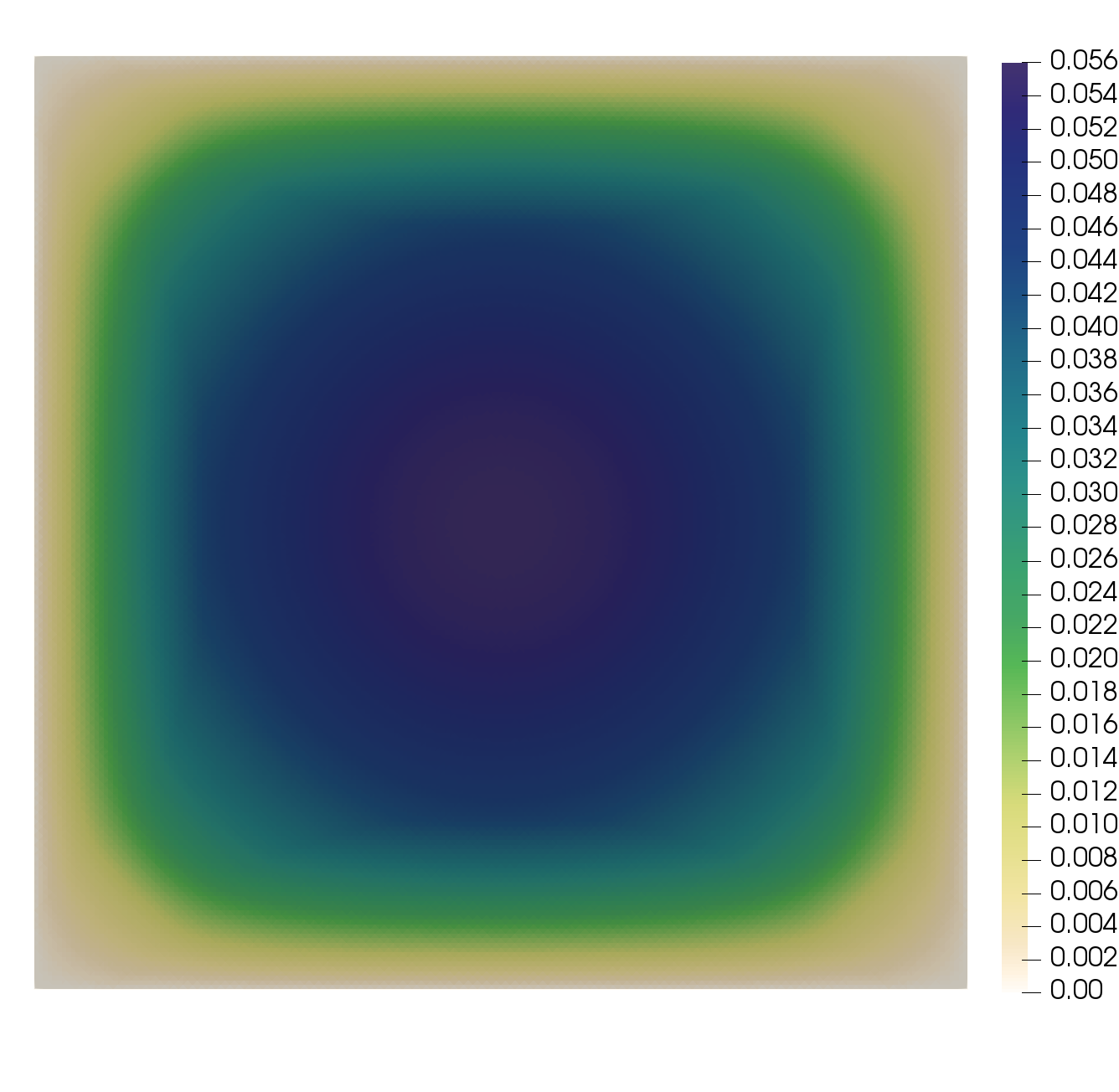}
\end{minipage}
\vskip-.5cm
\caption{Example 2. Optimal $a$ (left), and optimal $ u $ (right).}
\label{Fig:Ex3}
\end{figure}
\end{Exa}

\begin{Exa}
In the last case we show a numerical evidence of non existence of an optimal density $a_{opt}$ for problem of kind of (\ref{pbor}). We follow the counterexample showed in Section \ref{counter}. We consider $\Om=B(0,1)$ the unit ball in $\RR^2$ and the rest of data as in (\ref{Ex4}) with $\tau=0.23539$. We have solved the relaxed formulation of problem (\ref{pbor}) searching an optimal density $t_{opt}$ and an optimal matrix $A_{opt}$. Firstly, we have considered the problem with $\ep=0$. In this case, the computed optimal matrix is $A_{opt}=\mu_{t_{opt}} I$, a scalar matrix. We show in Figure \ref{Fig:Ex4} on the left the optimal value of $\mu_{t_{opt}}$ corresponding to $\ep=0$. As hoped, it agrees with the function $a_0$ defined by (\ref{up0a0}). On the other hand, we consider the problem with $\ep=0.5$. In this case the computed optimal matrix $A_{opt}$ is not scalar. We show in Figure \ref{Fig:Ex4} on the right the ratio $\lambda_1/\la_2$ with $\la_i, \,i=1,2$ the eigenvalues of $A_{opt}$. Observe that this ratio is not identically one, and therefore $A_{opt}$ is a non-isotropic matrix.

\begin{figure}[!h]
\begin{minipage}[t]{7.4cm}
\centering
\includegraphics[width=1.\textwidth]{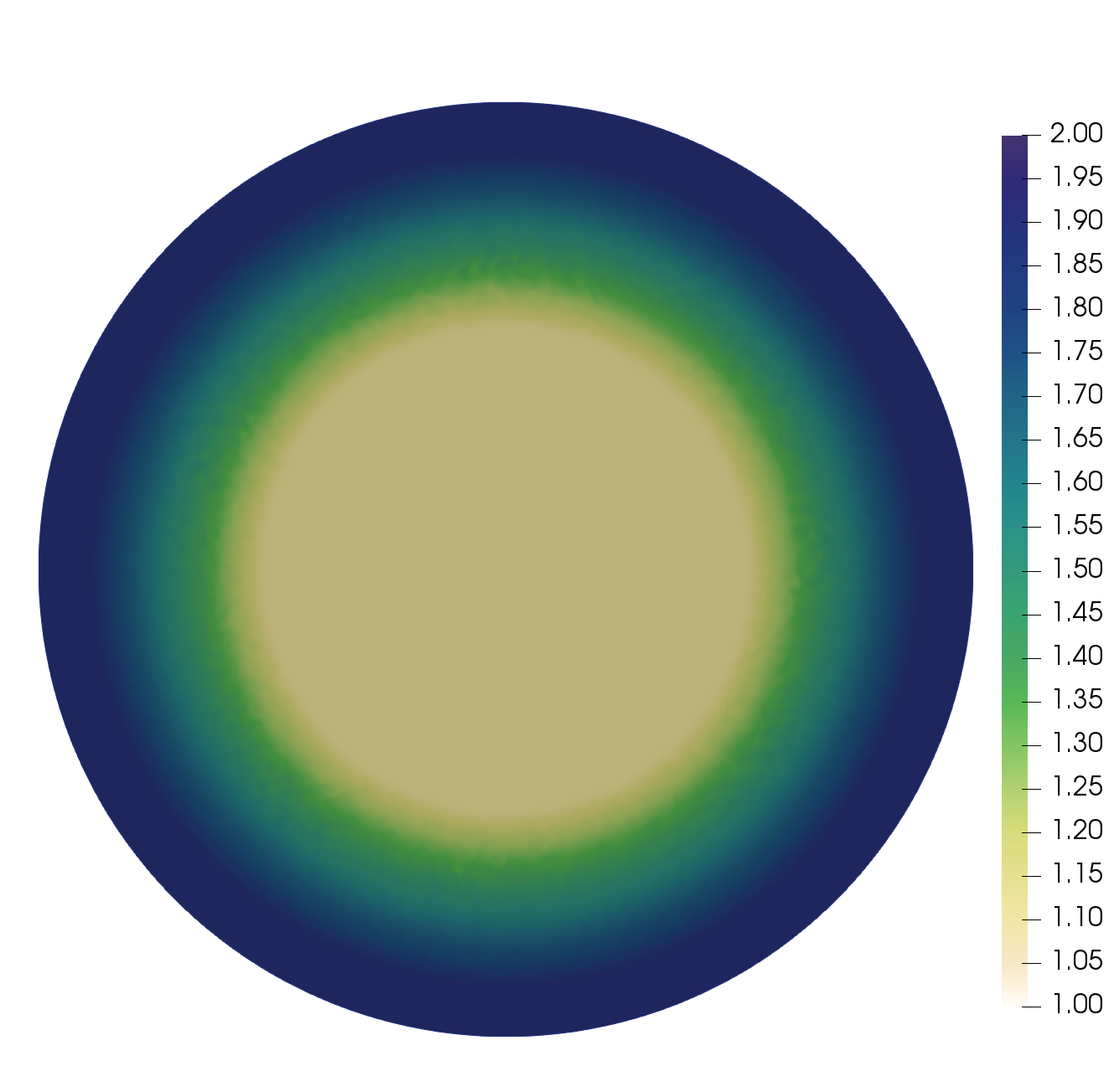}
\end{minipage}
\begin{minipage}[t]{7.4cm}
\centering
\includegraphics[width=1.\textwidth]{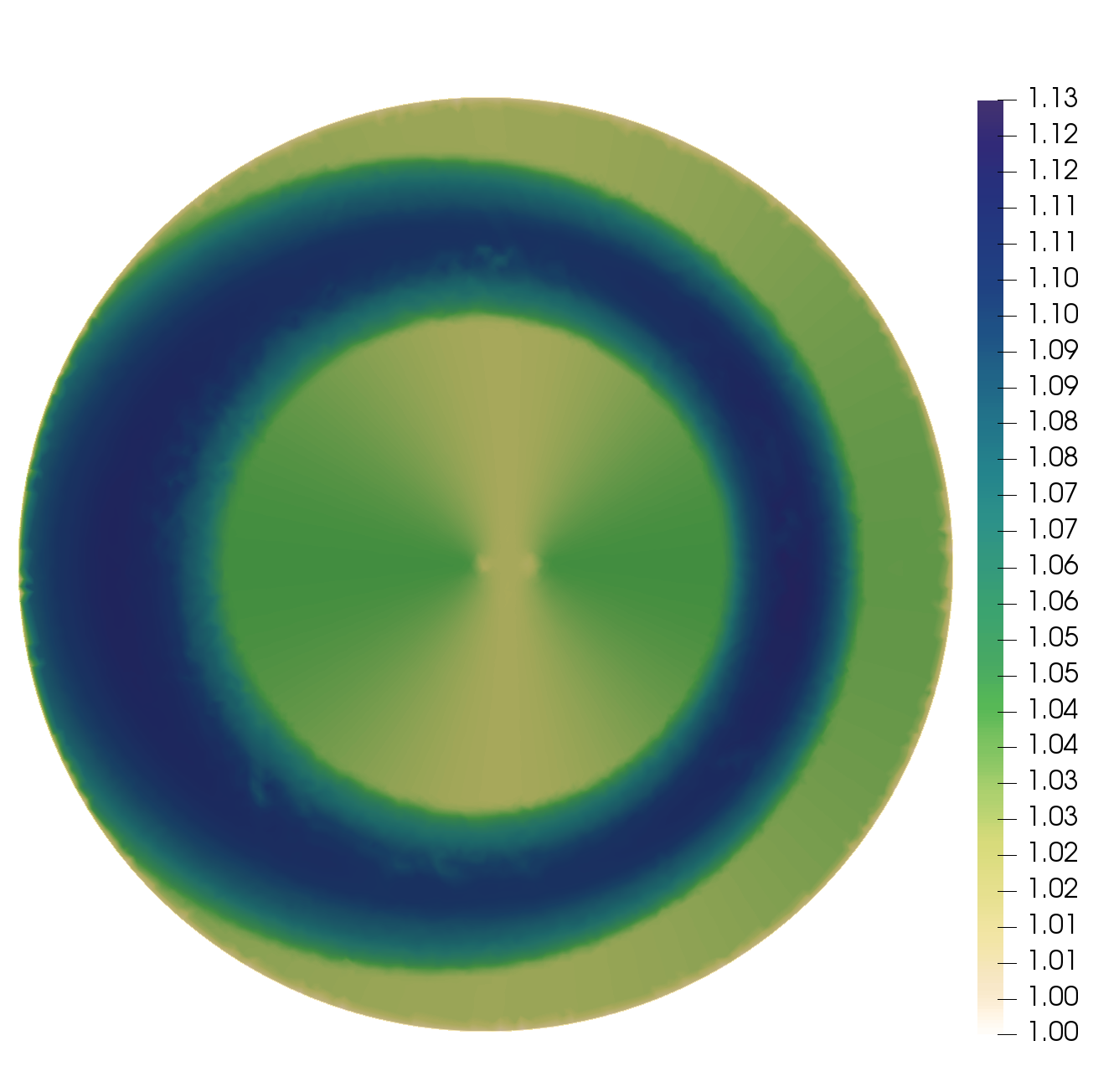}
\end{minipage}
\vskip-.5cm
\caption{Example 3, $\la_1=\la_2$ for $\ep=0$ (left), and $\la_2/\la_1$ for $\ep=0.5$.}
\label{Fig:Ex4}
\end{figure}
\end{Exa}

\noindent{\bf Acknowledgments. }The work of GB is part of the project 2017TEXA3H {\it``Gradient flows, Optimal Transport and Metric Measure Structures''} funded by the Italian Ministry of Research and University. GB is member of the Gruppo Nazionale per l'Analisi Matematica, la Probabilit\`a e le loro Applicazioni (GNAMPA) of the Istituto Nazionale di Alta Matematica (INdAM).\par 
The work of JCD and FM is a part of the FEDER project PID2020-116809GB-I00 of the {\it Ministerio de Ciencia e Innovaci\'on} of the government of Spain.
\bigskip


\end{document}